\documentclass[11pt]{article}

\title{\vspace{-50pt}Generating the Johnson filtration II: finite generation}
\author{Thomas Church\ and Andrew Putman\thanks{TC is supported in part by NSF grant DMS-1350138, the Alfred P.\ Sloan Foundation, the Institute for Advanced Study, and the Friends of the Institute.  AP is supported in part by NSF grant DMS-1737434.}\vspace{-6pt}}
\date{\today}

\usepackage{amsmath,amssymb,amsthm,amscd,amsfonts}
\usepackage{epsfig,pinlabel}
\usepackage[hmargin=1.25in, vmargin=1in]{geometry}
\usepackage[font=small,format=plain,labelfont=bf,up,textfont=it,up]{caption}
\usepackage{type1cm}
\usepackage{calc}
\usepackage{enumitem}
\usepackage{bm}
\usepackage{url}
\usepackage[all,cmtip]{xy}

\usepackage[%dvipdfm,%bookmarks, bookmarksdepth=2,
colorlinks=true, linkcolor=black, citecolor=black, urlcolor=black]{hyperref}

\newlist{compactitem}{itemize}{3}
\setlist[compactitem]{nosep}
\setlist[compactitem,1]{label=\textbullet}
\setlist[compactitem,2]{label=--}
\setlist[compactitem,3]{label=\ensuremath{\ast}}

\newlist{compactdesc}{description}{3}
\setlist[compactdesc]{nosep}

\newlist{compactenum}{enumerate}{3}
\setlist[compactenum]{nosep}
\setlist[compactenum,1]{label=\arabic*.}
\setlist[compactenum,2]{label=(\alph*)}
\setlist[compactenum,3]{label=\roman*.}

\usepackage{bookmark}
\bookmarksetup{depth=2}

\theoremstyle{plain}
\newtheorem{theorem}{Theorem}[section]
\newtheorem{maintheorem}{Theorem}

\newtheorem{proposition}[theorem]{Proposition}

\newtheorem{claim}{Claim}

\newtheorem{innerthm}{Theorem}
\newenvironment{theorem-prime}[1]{\renewcommand\theinnerthm{\ref*{#1}$'$}\innerthm}{\endinnerthm}
\theoremstyle{definition}

\theoremstyle{remark}
\newtheorem{remark}[theorem]{Remark}

\DeclareMathOperator{\Hom}{Hom}

\DeclareMathOperator{\Mod}{Mod}

\newcommand\Torelli{\ensuremath{{\mathcal I}}}
\newcommand\JohnsonKer{\ensuremath{{\mathcal K}}}
\newcommand\Johnson{\ensuremath{J}}
\DeclareMathOperator{\IA}{IA}
\DeclareMathOperator{\JIA}{JIA}
\DeclareMathOperator{\IO}{IO}
\DeclareMathOperator{\Sp}{Sp}
\DeclareMathOperator{\GL}{GL}
\DeclareMathOperator{\SL}{SL}

\newcommand\R{\ensuremath{\mathbb{R}}}

\newcommand\Z{\ensuremath{\mathbb{Z}}}
\newcommand\Q{\ensuremath{\mathbb{Q}}}
\newcommand\N{\ensuremath{\mathbb{N}}}

\renewcommand\P{\ensuremath{\mathbb{P}}}

\DeclareMathOperator{\HH}{H}

\DeclareMathOperator{\Aut}{Aut}
\DeclareMathOperator{\Out}{Out}

\newcommand{\para}[1]{\medskip\noindent\textbf{#1.}}
\newcommand{\into}{\hookrightarrow}

\newcommand{\bwedge}{\textstyle{\bigwedge}}
\newcommand{\abs}[1]{\left\lvert#1\right\rvert}
\newcommand{\normal}{\lhd}

\newcommand{\algp}{\mathcal{G}}
\newcommand{\iso}{\cong}
\DeclareMathOperator{\FInc}{{\tt FInc}}
\DeclareMathOperator{\BFInc}{{\bf FInc}}
\newcommand{\subgp}{K}
\newcommand{\FIncgroup}{$\FInc$-group}
\newcommand{\FIncgroups}{$\FInc$-groups}

\newcommand{\dual}{\vee}

\newcommand{\arXiv}[1]{\href{http://arxiv.org/abs/#1}{arXiv:#1}}
\newcommand{\myemail}[1]{\href{mailto:#1}{\nolinkurl{#1}}}

\begin{document}

\maketitle

\vspace{-14pt}
\begin{abstract}
We prove that every term of the lower central series and Johnson filtrations of the Torelli subgroups of the mapping
class group and the automorphism group of a free group are finitely generated in a stable range.  This was originally
proved for the commutator subgroup by Ershov--He.  \end{abstract}

\para{Historical note}
After we distributed a preliminary version of this paper, we learned that Ershov had independently proved
similar results.  The three of us decided to write a joint paper, and in the course of writing this joint paper
we managed to significantly improve the ranges of our theorems.  The combined paper is \cite{ChurchErshovPutman}.
The paper you are reading will be left as a permanent preprint and will not be submitted for publication.

\section{Introduction}

Let $\Sigma_g^b$ be a compact oriented genus $g$ surface with $b=0$ or $b=1$ boundary components. 
The {\em mapping class group}
of $\Sigma_g^b$, denoted $\Mod_g^b$, 
is the group of isotopy classes of orientation-preserving diffeomorphisms of $\Sigma_g^b$ that
fix $\partial \Sigma_g^b$ pointwise.  The group $\Mod_g^b$ acts on $\HH_1(\Sigma_g^b;\Z)$ and preserves
the algebraic intersection form.  Since $b \leq 1$, the algebraic intersection form is a nondegenerate symplectic
form, and thus this action induces a homomorphism $\Mod_g^b \rightarrow \Sp_{2g}(\Z)$ which is classically known
to be surjective.  The {\em Torelli group}, denoted $\Torelli_g^b$, is its kernel.  In summary,
we have a short exact sequence
\[1 \longrightarrow \Torelli_g^b \longrightarrow \Mod_g^b \longrightarrow \Sp_{2g}(\Z) \longrightarrow 1.\]
See~\cite{FarbMargalitPrimer} for a survey of the mapping class group and Torelli group.

\para{Lower central series}
Recall that if $G$ is a group, then the {\em lower central series} of $G$ is the sequence
\[G = G[1] \supset G[2] \supset G[3] \supset \cdots\]
of subgroups of $G$ defined via the inductive formula
\[G[1] = G \quad \text{and} \quad G[k+1] = [G[k],G] \quad \quad (k \geq 1).\]
Another way of saying this is that $G[k+1]$ is the smallest normal subgroup of $G$ such that $G/G[k+1]$ is
$k$-step nilpotent.  The lower central series of $\Torelli_g^b$ has connections to number theory
(see, e.g.,~\cite{Matsumoto}) and 3-manifolds (see, e.g.,~\cite{GaroufalidisLevine}).  Despite these connections,
the structure of the lower central series of $\Torelli_g^b$ is largely a mystery.  One
of the few large-scale structural results known about it is a theorem of Hain~\cite{HainInfinitesimal} giving
a finite presentation for the Lie algebra obtained from the lower central series of the Torelli group.

\para{Finiteness properties}
A classical theorem of Dehn~\cite{DehnGen} from 1938 says that $\Mod_g^b$ is finitely generated.  Since
$\Torelli_g^b$ is an infinite-index normal subgroup of $\Mod_g^b$, the naive guess would be
that $\Torelli_g^b$ is not finitely generated, and indeed McCullough--Miller~\cite{McCulloughMiller}
proved that $\Torelli_2^b$ is not finitely generated.  However, a deep and surprising theorem of
Johnson~\cite{JohnsonFinite} says that $\Torelli_g^b$ is finitely generated for $g \geq 3$.  Another
paper of Johnson~\cite{JohnsonAbel} proves that $\Torelli_g^b[2]$ is commensurable with the
{\em Johnson kernel} subgroup $\JohnsonKer_g^b$, that is, the subgroup of $\Mod_g^b$ generated
by Dehn twists about simple closed separating curves.  The group $\JohnsonKer_g^b$ is the subject
of a large literature; in particular, a recent deep paper of Dimca--Papadima~\cite{DimcaPapadimaKg}
proves that $\HH_1(\JohnsonKer_g^b;\Q)$ is finite dimensional for $g \geq 4$.  The group
$\HH_1(\JohnsonKer_g^b;\Q)$ was later computed for $g \geq 6$ by Dimca--Hain--Papadima~\cite{DimcaPapadimaHainKg}.

\para{Finite generation}
In a very recent breakthrough,
Ershov--He~\cite{ErshovHe} prove the following.
\begin{compactitem}
\item $[\Torelli_g^b,\Torelli_g^b]$ and $\JohnsonKer_g^b$ are finitely generated if $g \geq 5$
%\item $[\Torelli_g^b,\Torelli_g^b]$ is finitely generated if $g \geq 5$.
\item For any subgroup $\subgp\subset \Torelli_g^b$ containing $\Torelli_g^b[k]$, the abelianization $\HH_1(\subgp;\Z)$ is finitely generated for $k \geq 2$ and $g \geq 12(k-1)$.
\end{compactitem}
Our first main theorem extends the results of Ershov--He by proving that such a subgroup $\subgp$ is in fact finitely generated.

\begin{maintheorem}
\label{maintheorem:torellilcs}
For $b=0$ or $b=1$ and any $k\geq 2$, let $\subgp \subset \Torelli_g^b$ be a subgroup such that $\Torelli_g^b[k] \subset \subgp$.  
Then $\subgp$ is finitely generated
as long as $g\geq 6k-4$.
\end{maintheorem}
In particular, the $k^{\text{th}}$ term of the lower central series $\Torelli_g^b[k]$ is itself finitely generated for $g \geq 6k-4$. We show in \S\ref{section:reductions} that the theorem reduces to the case of $\Torelli_g^b[k]$ itself.

\begin{remark}
As we mentioned above, Ershov--He proved the special case of Theorem~\ref{maintheorem:torellilcs} when $k=2$. Our proof of Theorem~\ref{maintheorem:torellilcs} is self-contained, so it does not \emph{depend} on the results of Ershov--He. Rather, in the case $k=2$ our proof reduces to an argument essentially equivalent to the original Ershov--He proof. The same comment applies to Theorem~\ref{maintheorem:ialcs} below.
\end{remark}

\begin{remark}
After distributing a preliminary version of this paper, we learned that Ershov--He have independently proven a result
that is almost identical to Theorem \ref{maintheorem:torellilcs} (and similarly for Theorem~\ref{maintheorem:ialcs} below).
\end{remark}

\para{The Johnson filtration}
We want to highlight an important special case of Theorem \ref{maintheorem:torellilcs}.
Fix some $g \geq 0$ and $b \leq 1$.  Pick a basepoint $\ast \in \Sigma_g^b$; if $b=1$, then
choose $\ast$ such that it lies in $\partial \Sigma_g^b$.  Define $\pi = \pi_1(\Sigma_g^b,\ast)$.  Since
$\Mod_g^1$ is built from diffeomorphisms that fix $\partial \Sigma_g^1$ (and thus fix $\ast$), there
is a homomorphism $\Mod_g^1 \rightarrow \Aut(\pi)$.  For closed surfaces, there is no fixed basepoint, so
we only obtain a homomorphism $\Mod_g \rightarrow \Out(\pi)$.  In both cases, this action preserves
the lower central series of $\pi$, so we obtain homomorphisms
\[\psi_g^1[k] \colon \Mod_g^1 \rightarrow \Aut(\pi/\pi[k]) \quad \text{and} \quad \psi_g[k]\colon \Mod_g \rightarrow \Out(\pi/\pi[k]).\]
The $k^{\text{th}}$ term of the {\em Johnson filtration} on $\Mod_g^b$, denoted $\Johnson_g^b(k)$, is the kernel
of $\psi_g^b[k+1]$.  This filtration was introduced by Johnson in 1981~\cite{JohnsonSurvey}.

Chasing the definitions, we find that $\Johnson_g^b(1) = \Torelli_g^b$.  Moreover,
Johnson~\cite{JohnsonKer} proved that $\Johnson_g^b(2) = \JohnsonKer_g^b$.
It is easy to see that $\Torelli_g^b[k]\subset\Johnson_g^b(k)$ for all $k$, but these filtrations are
known not to be equal.  In fact, Hain~\cite{HainInfinitesimal} proved that they even define inequivalent topologies
on $\Torelli_g^b$.  Since $\Torelli_g^b[k]\subset\Johnson_g^b(k)$, the following result is a special case of Theorem \ref{maintheorem:torellilcs}.

\begin{maintheorem}
\label{maintheorem:torellijohnson}
For $b=0$ or $b=1$ and any $k\geq 2$, the $k^{\text{th}}$ term of the Johnson filtration $\Johnson_g^b(k)$ is finitely generated for $g \geq 6k-4$.
\end{maintheorem}

\para{Automorphism groups of free groups}
Let $F_n$ be a free group on $n$ generators.  We also have a version of Theorems~\ref{maintheorem:torellilcs} and \ref{maintheorem:torellijohnson} for
$\Aut(F_n)$.  The group $\Aut(F_n)$ acts on $F_n^{\text{ab}} = \Z^n$.  The kernel of this action is the
{\em Torelli subgroup} of $\Aut(F_n)$ and is denoted $\IA_n$. A classical theorem of 
Magnus~\cite{MagnusGenerators} from 1935 shows that $\IA_n$ is finitely generated for all $n$.
Building on the aforementioned work of Dimca--Papadima~\cite{DimcaPapadimaKg} for the mapping class group,
Papadima--Suciu~\cite{PapadimaSuciuKg} proved that $\HH_1(\IA_n[2];\Q)$ is finite-dimensional
for $n \geq 5$.  Ershov--He's breakthrough paper~\cite{ErshovHe} also applies to $\IA_n$, proving the following.
\begin{compactitem}
\item $[\IA_n,\IA_n]$ is finitely generated if $n \geq 4$.
\item For any subgroup $\subgp\subset \IA_n$ containing $\IA_n[k]$, the abelianization $\HH_1(\subgp;\Z)$ is finitely generated if $k \geq 2$ and $n \geq 12(k-2)$.
\end{compactitem}

We extend the results of Ershov--He by proving that such a subgroup $\subgp$ is in fact finitely generated.

\begin{maintheorem}
\label{maintheorem:ialcs}
For any $k\geq 2$, let $\subgp \subset \IA_n$ be a group such that $\IA_n[k] \subset \subgp$.
Then $\subgp$ is finitely generated as long as $n\geq 6k-4$.
\end{maintheorem}

In particular, $\IA_n[k]$ is itself finitely generated for $n \geq 6k-4$. We show in \S\ref{section:reductions} that the theorem reduces to the case of $\IA_n[k]$ itself.

\begin{remark}
One can also consider the Torelli subgroup $\IO_n$ of $\Out(F_n)$.  The homomorphism $\IA_n \rightarrow \IO_n$
is surjective, so Theorem \ref{maintheorem:ialcs} also implies a similar result for the lower central
series of $\IO_n$.
\end{remark}

\para{Johnson filtration for automorphism group of free group}
Just like for the mapping class group, there is a natural homomorphism
\[\psi_n[k]\colon \Aut(F_n) \rightarrow \Aut(F_n / F_n[k]).\]
The $k^{\text{th}}$ term of the {\em Johnson filtration} for $\Aut(F_n)$, denoted
$\JIA_n(k)$, is the kernel of ${\psi_n[k+1]}$.  This filtration was actually introduced
by Andreadakis~\cite{Andreadakis} in 1965, much earlier than the Johnson filtration for
the mapping class group.  It is well-known that $\IA_n[k] \subset \JIA_n(k)$.  However, 
it is an open problem whether or not these two filtrations are equal (or at least
commensurable).  In any case, since $\IA_n[k] \subset \JIA_n(k)$, the following result is a special case of Theorem \ref{maintheorem:ialcs}.

\begin{maintheorem}
\label{maintheorem:iajohnson}
For any $k\geq 2$, the group $\JIA_n(k)$ is finitely generated for $n \geq 6k-4$.
\end{maintheorem}

\para{$\BFInc$-groups}
We give an easy reduction in \S\ref{section:reductions} showing that to prove Theorem~\ref{maintheorem:torellilcs}, it suffices to prove that $\Torelli_g^1[k]$ is finitely generated for $g \geq 6k-4$. Similarly, we show that to prove Theorem~\ref{maintheorem:ialcs}, it suffices to prove that $\IA_n[k]$ is finitely generated for $n \geq 6k-4$.
We  prove both these results
as consequences of a single abstract theorem. This result is essentially a structure theorem applying to the FI-groups and weak FI-groups appearing in our earlier paper~\cite{ChurchPutmanGenJohnson}. In essence, it says that if a weak FI-group satisfying certain additional conditions is finitely generated in finite degree, then its lower central series is finitely generated.

However, the technicalities of weak FI-groups often make arguments quite difficult to read.
To simplify the presentation of the argument, we have written our theorem in terms of a simpler structure, namely  \FIncgroups{}, which we now define. We will not mention FI-groups or weak FI-groups in the rest of the paper (with the exception of the tangential Remarks~\ref{remark:gendegree} and~\ref{remark:FInc-vs-FI}), so the reader need not be familiar with these at all.

Define $\N = \{1,2,3,\ldots\}$.  Let $\FInc$ be the category whose
objects are finite subsets of $\N$ and whose morphisms are inclusions.  An
{\em \FIncgroup{}} is a functor from $\FInc$ to the category of groups.  More
concretely, an \FIncgroup{} $G$ consists of the data of a group $G_I$ for each
finite $I \subset \N$ together with homomorphisms $G_I^J\colon G_I \rightarrow G_J$ whenever
$I \subset J$.  These homomorphisms must satisfy $G_J^K \circ G_{I}^J = G_I^K$ whenever $I \subset J \subset K$.

\para{Structure theorem}
Stating our theorem requires some preliminary definitions. 
Given an \FIncgroup{} $G$, for finite $I \subset J \subset \N$ we will denote by $G_I(J)$ the image
of the map $G_I^J \colon G_I \rightarrow G_J$.  For 
$n \in \N$, let $[n] = \{1,\ldots,n\} \subset \N$. We will write $G_n$ and $G_n(I)$ instead
of $G_{[n]}$ and $G_{[n]}(I)$.

We will say that an \FIncgroup{} $G$ is a {\em commuting \FIncgroup{}} if for all finite $K \subset \N$ and all disjoint $I,J \subset K$, there
is some $G_K$-conjugate of $G_K(I) \subset G_K$ that commutes with $G_K(J)$.

Finally, we say that an \FIncgroup{} $G$ is {\em generated in degree $d$} if for all $n\geq d$, the group $G_n$ is generated by the $\binom{n}{d}$ subgroups  $G_n(I)$ for $I\in \binom{[n]}{d}$. (Throughout the paper,  $\binom{X}{d}$ denotes the set of $d$-element subsets of $X$.)

\begin{remark}
\label{remark:gendegree}
Readers familiar with FI-modules from~\cite{ChurchEllenbergFarbFI} or weak FI-groups from~\cite{ChurchPutmanGenJohnson} might expect this to be equivalent to the condition that for all $K$ with $\abs{K}\geq d$, the group $G_K$ is generated by the subgroups  $G_K(I)$ for $I\in \binom{K}{d}$. This is \emph{not} actually equivalent, and so we have taken the weaker definition (since it suffices for our main theorem). It also makes it easier to verify that $\Torelli$ is generated in degree 3 (compare our proof in \S\ref{section:verify} with the remark preceding \cite[Lemma 4.6]{ChurchPutmanGenJohnson}, which is no longer necessary with this definition).
\end{remark}

Our main technical theorem is the following structure theorem.
\begin{maintheorem}
\label{maintheorem:FIncgroup}
Let $G$ be a commuting \FIncgroup{} such for each $n$ there exists a connected semisimple $\R$-algebraic group without compact factors $\algp_n$, a lattice $\Gamma_n\subset \algp_n(\R)$, and an action of $\Gamma_n$ on $G_n$ by outer automorphisms such that the following two conditions hold.
\begin{compactitem}
\item For any subsets $K,K'\subset [n]$ with $\abs{K}=\abs{K'}$, there exists some $\gamma\in \Gamma_n$ such that $\gamma\cdot G_n(K)$ is conjugate to $G_n(K')$ inside $G_n$.
\item The induced action of $\Gamma_n$ on $\HH_1(G_n;\R)$ extends to an algebraic representation of $\algp_n$.
\end{compactitem}
Suppose that $G$ is generated 
in degree $d$ as an \FIncgroup{} and $G_d$ is finitely generated.
Then for each $k\geq 1$ the group $G_n[k]$ is finitely generated for all  $n\geq (2k-1)d$.
\end{maintheorem}

We will prove that both $\Torelli_g^1$ and $\IA_n$ can be endowed with the structure of \FIncgroups{} $\Torelli$ and $\IA$ satisfying
the above conditions with $d=3$. As stated, Theorem~\ref{maintheorem:FIncgroup} would yield the range $n\geq (2k-1)d=6k-3$. However, we give a slight strengthening in Theorem~\ref{maintheorem:stronger}, where a minor condition (satisfied by both $\Torelli$ and $\IA$) allows us to improve this range by 1 to $n\geq 6k-4$, yielding Theorems \ref{maintheorem:torellilcs} and \ref{maintheorem:ialcs}.

\begin{remark}
\label{remark:FInc-vs-FI}
Since Theorem~\ref{maintheorem:FIncgroup} is stated in terms of \FIncgroups{}, it would be natural to think  that it will apply to a much wider class than the weak FI-groups of \cite{ChurchPutmanGenJohnson}. We believe this would be a mistake: the technical conditions we impose in Theorem~\ref{maintheorem:FIncgroup} force our \FIncgroups{} to behave very much like weak FI-groups. Although we no longer explicitly require the permutation group $S_n$ to act on $G_n$, the action of $\Gamma_n$ required in Theorem~\ref{maintheorem:FIncgroup} fulfills much the same role (for example, it forces the subgroups $G_n(K)$ and $G_n(K')$ to be isomorphic). Not only do weak FI-groups provide all the examples of \FIncgroups{} that we will apply our structure theorem to, it is currently hard for us to imagine \emph{any} $\FInc$-group it applies to that is not actually a weak FI-group. So at present it seems that this is only a pedagogical advance, and Theorem~\ref{maintheorem:FIncgroup} is ``really'' a theorem about weak FI-groups (but we would be glad to be proved wrong!).
\end{remark}

\para{Comments on the proof}
In their paper~\cite{ErshovHe}, Ershov--He introduced the beautiful idea of proving that $[\Torelli_g^1,\Torelli_g^1]$
is finitely generated by exploiting a powerful result of Brown~\cite{BrownBNS} that relates the BNS-invariant
of a group to the existence of certain actions on $\R$-trees.  These actions on $\R$-trees can be understood
using the fact that $\Torelli_g^1$ contains large families of commuting elements. They also introduced the idea of using an action of $\SL_n\Z$ to guarantee these commuting elements ``survive'' in an appropriate sense. We will use these same
ideas in our proof of Theorem \ref{maintheorem:FIncgroup}, though we will handle the ideas quite differently.

\para{Outline}
We begin in \S \ref{section:brown} by discussing Brown's criterion.  Next, in \S \ref{section:abstract}
we will prove Theorem~\ref{maintheorem:FIncgroup} (in its strengthened form of Theorem~\ref{maintheorem:stronger}).  Finally, in \S \ref{section:examples} we will verify that this theorem applies to $\Torelli_g^1$ and $\IA_n$, and thus prove Theorems \ref{maintheorem:torellilcs}~and~\ref{maintheorem:ialcs}.

\para{Acknowledgments}
We wish to thank Mikhail Ershov, Benson Farb, Sue He, Dan Margalit, and John Meier for helpful comments.

\section{Brown's criterion for finite generation}
\label{section:brown}

If $V$ is a real vector space, then we will write $\P^+V$ for the positive projectivation $\P^+V={(V-\{0\})/\R^\times_+}$.
If $G$ is a finitely generated group, let $S(G)$ be the ``character sphere'' $S(G)=\P^+\Hom(G,\R)$. If $N\normal\, G$ has abelian quotient $G/N$, the quotient map induces an embedding $\Hom(G/N,\R)\into \Hom(G,\R)$, so we may consider $S(G/N)$ as a canonical subset of $S(G)$ consisting of those homomorphisms $G\to \R$ factoring through $G/N$. 

The Bieri--Neumann--Strebel invariant $\Sigma(G)$ is an open subset of $S(G)$ which is invariant under any automorphism of $G$. %\cite[Corollary~1.4]{KMM-BNS}. 
It has the key property that if $G$ is finitely generated and $G/N$ is abelian, then \[\text{$N$ is finitely generated\qquad $\iff$ \qquad $\Sigma(G)\supset S(G/N)$.}\] 

In the same journal issue where Bieri--Neumann--Strebel defined $\Sigma(G)$, Brown~\cite{BrownBNS} gave a necessary and sufficient criterion describing which $[\lambda]\in S(G)$ belong to the BNS invariant $\Sigma(G)$ in terms of certain actions on $\R$-trees. We will make use of a simpler sufficient condition which is a consequence of Brown's criterion. We use a formulation of this condition given by Koban--McCammond--Meier~\cite{KMM-BNS}, though it goes back further (compare with Meier--VanWyk~\cite[Theorem~6.1]{MeierVanWyk}). A variant of this condition was used by Ershov--He (see~\cite[Proposition~3.4]{ErshovHe}).

Given any subset $A\subset G$, let $C(A)$ denote the ``commutation graph'' whose vertices are the elements of $A$, and where $a$ and $a'$ are connected by an edge if $a$ and $a'$ commute.
We say that a subset $A\subset G$ survives under $\lambda\colon G\to \R$ if $\lambda(a)\neq 0$ for all $a\in A$. For $A,B\subset G$ we say that $A$ dominates $B$ if every $b\in B$ commutes with some $a\in A$.

\begin{proposition}[{\cite[Lemma~1.9]{KMM-BNS}}]
\label{prop:KMM}
Let $G$ be a finitely generated group and $\lambda\colon G\to \R$ a nonzero homomorphism. Suppose there exist subsets $A\subset G$ and $B\subset G$ such that $A$ survives under $\lambda$, $C(A)$ is connected, $A$ dominates $B$, and $B$ generates $G$. Then $[\lambda]\in \Sigma(G)$.
\end{proposition}

\section{Structure theorem}
\label{section:abstract}
Our goal in this section is to prove Theorem~\ref{maintheorem:FIncgroup}. We will in fact prove the following slightly stronger theorem, which simply restates Theorem~\ref{maintheorem:FIncgroup}
and adds the additional condition \ref{part:improverangebyone}. We emphasize that the \FIncgroups{} $\Torelli$ and $\IA$ to which we will apply the structure theorem in \S\ref{section:examples} both do satisfy the condition \ref{part:improverangebyone}.
\begin{theorem-prime}{maintheorem:FIncgroup}
\label{maintheorem:stronger}
Let $G$ be a commuting \FIncgroup{} such for each $n$ there exists a connected semisimple $\R$-algebraic group without compact factors $\algp_n$, a lattice $\Gamma_n\subset \algp_n(\R)$, and an action of $\Gamma_n$ on $G_n$ by outer automorphisms such that the following two conditions hold.
\begin{compactitem}
\item For any subsets $K,K'\subset [n]$ with $\abs{K}=\abs{K'}$, there exists some $\gamma\in \Gamma_n$ such that $\gamma\cdot G_n(K)$ is conjugate to $G_n(K')$ inside $G_n$.
\item The induced action of $\Gamma_n$ on $\HH_1(G_n;\R)$ extends to an algebraic representation of $\algp_n$.
\end{compactitem}
Suppose that $G$ is generated 
in degree $d$ as an \FIncgroup{} and $G_d$ is finitely generated.  The following then hold.
\begin{compactenum}[label=(\roman*)]
\item\label{part:maintheorem} For each $k\geq 1$ the group $G_n[k]$ is finitely generated for all  $n\geq (2k-1)d$.
\item\label{part:improverangebyone} Suppose that $d\geq 2$ and that $G$ has the additional property that for all $m\leq n$,
\begin{equation}
\tag{E$'$-ii}
\label{eq:strongercommuting}
G_n([m])\text{ commutes with }G_n(I)\text{ for all $I\in \binom{[n]}{m}$ with $I\cap [m]=\emptyset$}.
\end{equation}
Then for each $k\geq 2$ the group $G_n[k]$ is finitely generated for all  $n\geq (2k-1)d-1$.
\end{compactenum}
\end{theorem-prime}

\begin{remark} The hypotheses here can be relaxed. For example, if we assume instead that $G$ is generated in degree $d$ as an \FIncgroup{} and that $G_m$ is finitely generated for some $m>d$, the same proof will shows that $G_n[k]$ is finitely generated for all $n\geq \max(m,(2k-1)d)$. This indicates something about the structure of the proof\,---\,namely that we never need to do induction on $n$.

We could also replace the condition that $\gamma\cdot G_n(K)$ is conjugate to $G_n(K')$ with the much weaker assumption that $\gamma$ takes the image of $\HH_1(G_n(K);\R)$ in $\HH_1(G_n;\R)$ to the image of $\HH_1(G_n(K');\R)$. This weaker assumption leaves open the possibility that the subgroups $G_n(K)$ and $G_n(K')$ might not even be isomorphic, so in theory this allows much greater generality. However we are not aware of any examples that this would apply to.
\end{remark}

The proof of Theorem~\ref{maintheorem:stronger}  occupies the remainder of this section. 
The proof will go by induction (though it is less inductive than one might expect). Before beginning the induction, we establish some basic facts that will not depend on the inductive hypothesis.

If $G$ is an \FIncgroup{}, then for any $k$ the lower central series subgroups $G_I[k]\subset G_I$ satisfy $G_I^J(G_I[k])\subset G_J[k]$. Thus they define an \FIncgroup{} which we may denote $G[k]$. Note that for $I\subset J$, all three of $G_J(I)[k]$ and $G_J[k](I)$ and $G[k]_J(I)$ describe the same group (namely the image of $G_I[k]$ in $G_J$), so we need not worry about the distinction.

\begin{claim}
\label{claim:Gknormalgen}
For all $n\geq dk$, the group $G_n[k]$ is generated by the $G_n$-conjugates of $G_n[k](I)$ for $I\in \binom{[n]}{dk}$.
\end{claim}\begin{proof}
Suppose that $n\geq dk$, and let $S=\bigcup_{I\in \binom{[n]}{d}} G_n(I)$. Our assumption that $G$ is generated in degree $d$ means precisely that $S$ generates $G_n$. For any group $G_n$ and any generating set $S$, the lower central series $G_n[k]$ is $G_n$-normally generated by the length-$k$ commutators $[s_1,[s_2,\ldots,[s_{k-1},s_k]]]$ with $s_i\in S$ (compare with~\cite[Eq.\ (2-8)]{ChurchPutmanGenJohnson}). Given such generators with $s_i\in G_n(I_i)$, let $J=\bigcup I_i$. Since $s_i\in G_n(I_i)\subset G_n(J)$, the commutator $[s_1,[s_2,\ldots,[s_{k-1},s_k]]]$ belongs to $G_n(J)[k]$. Since $\abs{J}\leq \sum_{i=1}^k \abs{I_i}= dk$, there is some $J'\in \binom{[n]}{dk}$ such that $[s_1,[s_2,\ldots,[s_{k-1},s_k]]]\in G_n(J')[k]$, verifying the claim.
\end{proof}

Let $L_I=\bigoplus_{k\geq 1}L_I[k]$ be the real Lie algebra associated to the lower central series of the group $G_I$, i.e.\ $L_I[k]=(G_I[k]/G_I[k+1])\otimes \R$ with bracket induced from the commutator. The \FIncgroup{} $G[k]$ descends to an \FIncgroup{} $L[k]$.
\begin{claim}
\label{claim:Lkgen}
The \FIncgroup{} $L[k]$ is generated in degree $dk$.
\end{claim} 
\begin{proof} Suppose that $n\geq dk$. Claim~\ref{claim:Gknormalgen} states that $G_n[k]$ is generated by the $G_n$-conjugates of $G_n[k](J)$ for $J\in \binom{[n]}{dk}$. Since conjugation by $G_n$ becomes trivial in $G_n[k]/G_n[k+1]$, this implies that $G_n[k]/G_n[k+1]$ is \emph{generated} by the images of $G_n[k](J)$ for $J\in \binom{[n]}{dk}$. Thus the \FIncgroup{} $G[k]/G[k+1]$ is generated in degree $dk$, and thus so is $L[k]$.
\end{proof}

The outer action of $\Gamma_n$ on $G_n$ preserves the characteristic subgroup $G_n[k]$ and descends to an honest action of $\Gamma_n$ on $L_n[k]$.

\begin{claim}
\label{claim:tiltgamma}
For any $n\geq dk$ and any nonzero homomorphism $\lambda\in L_n[k]^\dual$, there exists $\gamma\in \Gamma_n$ such that $\gamma^*\lambda\in L_n[k]^\dual$ has the property that for every $I\in\binom{[n]}{dk}$ the restriction $\gamma^*\lambda|_{L_n[k](I)}$ is nonzero.
\end{claim}
\begin{proof}
Note that the subspaces $L_n[k](I)$ for all $I\in \binom{[n]}{dk}$ together span $L_n[k]$ by Claim~\ref{claim:Lkgen}, so the restriction $\lambda|_{L_n[k](I_1)}$ of $\lambda$ itself must be nonzero for \emph{some} $I_1\in \binom{[n]}{dk}$. But this is not enough.

We begin by showing that the action of $\Gamma_n$ on $L_n[k]$ extends to an algebraic representation of $\algp_n(\R)$. The Lie algebra $L_n=\bigoplus L_n[k]$ is always generated as a Lie algebra in degree~$1$, i.e.\ by $L_n[1]=(G_n[1]/G_n[2])\otimes \R=\HH_1(G_n;\R)$. In other words, the natural map
\[L_n[1]^{\otimes k}=\HH_1(G_n;\R)^{\otimes k}\to L_n[k]\] is surjective. By assumption, the action of $\Gamma_n$ on $L_n[1]^{\otimes k}$ extends to an algebraic action of $\algp_n$, and we would like to show it descends to an action of $\algp_n$ on  $L_n[k]$.  The map $L_n[1]^{\otimes k}\to L_n[k]$ is equivariant for the action of $\Gamma_n$, so its kernel is $\Gamma_n$-invariant. Since $\Gamma_n$ is a lattice inside the real points $\algp_n(\R)$ of a connected semisimple algebraic group without compact factors, $\Gamma_n$ is Zariski-dense in $\algp_n$ by Borel density~\cite{BorelDensity}. Thus the kernel of $L_n[1]^{\otimes k}\to L_n[k]$ is in fact $\algp_n$-invariant, so the action of $\algp_n$ descends to $L_n[k]$ as desired.

Let $V=L_n[k]$ and $V_I=L_n[k](I)$.
For each $I\in \binom{[n]}{dk}$, the condition $\mu|_{V_I}=0$ defines a closed subvariety $X_I$ of $V^\dual=\{\mu\colon V\to \R\}$. The claim asserts that there exists $\gamma\in \Gamma_n$ such that $\gamma^*\lambda\notin X_I$ for all $I\in \binom{[n]}{dk}$. Suppose that this is not true. This means that for all $\gamma\in \Gamma_n$ we have $\gamma^*\lambda\in \bigcup_I X_I$. In other words the $\Gamma_n$-orbit $\Gamma_n\cdot \lambda$ is contained in $\bigcup_I X_I$. Since $\Gamma_n$ is Zariski-dense in $\algp_n$, this implies that $\algp_n\cdot \lambda\subset \bigcup_I X_I$. Since $\algp_n$ is connected, $\algp_n$ (and thus its image $\algp_n\cdot \lambda$) is irreducible, so it must be contained in one irreducible factor: there exists some $I_0\in \binom{[n]}{dk}$ such that $\algp_n\cdot \lambda \subset X_{I_0}$. In particular, for this particular $I_0$ every $\gamma\in \Gamma_n$ has $\gamma^*\lambda|_{V_{I_0}}=0$.

Recall from the beginning of the proof that $\lambda|_{V_{I_1}}$ is nonzero for some $I_1\in \binom{[n]}{dk}$.
But one of the hypotheses of the theorem was that if $\abs{I_0}=\abs{I_1}$ there exists some $g\in \Gamma_n$ such that $g\cdot G_n(I_1)$ is conjugate to $G_n(I_0)$. This implies that $g\cdot V_{I_1}=V_{I_0}$, so \[g^*\lambda|_{V_{I_0}}=0 \iff \lambda|_{V_{I_1}}=0\] This contradiction concludes the proof of the claim.
\end{proof}

\begin{claim}
\label{claim:handlecomplex}
Given $n$ and $m\geq 1$, let $Y_{n,m}$ be the graph whose vertices are pairs $(I\in \binom{[n]}{m},$ $\alpha\in G_n)$ where there is an edge between $(I,\alpha)$ and $(I',\beta)$ if the conjugate subgroups $G_n(I)^{\alpha}$ and $G_n(I')^{\beta}$ commute. Then $Y_{n,m}$ is connected for all $n\geq 2m+d$.
If $d\geq 2$ and $G$ satisfies the additional condition~\eqref{eq:strongercommuting}, then $Y_{n,m}$ is connected for all $n\geq 2m+d-1$.
\end{claim}
\begin{proof}
Suppose either that $n\geq 2m+d$ or that $d\geq 2$ and $n\geq 2m+d-1$; note that in either case we have $n\geq 2m+1$ and $n\geq d$.

We remark that a great many of the conjugates $G_n(I)^\alpha$ will of course coincide, but we do not need to keep track of this.
$G_n$ acts on $Y_{n,m}$ by $(I,\alpha)\mapsto (I,\alpha g)$. 

Fix one basepoint $v=(I_0,1)$. We will apply the condition of~\cite{PutmanConnectivityNote} to prove that $Y_{n,m}$ is connected. According to \cite[Lemma~2.1]{PutmanConnectivityNote} it suffices to show that
\begin{compactenum}[label=(\roman*)]
\item\label{part:orbitconnect} for every $G_n$-orbit of vertices, there is a path from $v$ to some vertex in that orbit, and
\item\label{part:generatorconnect} there is a generating set $S=S^{\pm}$ for $G_n$ such that for every $s\in S$, there is a path from $v$ to $s\cdot v$.
\end{compactenum}
To prove \ref{part:orbitconnect}, consider the relation on $\binom{[n]}{m}$ defined by $I\sim I'$ if there is a path in $Y_{n,m}$ from $(I,1)$ to $(I',\alpha)$ for some $\alpha\in G_n$. Note that this relation is automatically an equivalence relation, and \ref{part:orbitconnect} asserts precisely that $I_0\sim I$ for all $I$. If $I,I'\in \binom{[n]}{m}$ have $I\cap I'=\emptyset$, then the definition of commuting \FIncgroup{} states that there is a $G_n$-conjugate of $G_n(I')$ which commutes with $G_n(I)$, or in other words that $I\sim I'$ whenever $I\cap I'=\emptyset$. Since $n\geq 2m+1$, between any two $m$-element subsets $I_0$ and $I$ there exists a sequence $I_0,I_1,\ldots,I_\ell=I$ where $I_i\in \binom{[n]}{m}$ and $I_i\cap I_{i+1}=\emptyset$. This verifies condition \ref{part:orbitconnect}.

For \ref{part:generatorconnect}, as in the proof of Claim~\ref{claim:Gknormalgen} we choose the generating set $S=\bigcup_{J\in \binom{[n]}{d}} G_n(J)$. Since $n\geq d$, we know that $S$ generates $G_n$. Consider some generator $s\in S$ belonging to $G_n(J)$ for some $J\in\binom{[n]}{d}$. We consider various cases, but in each case we will find some $(I,\alpha)$ such that $G_n(I)^\alpha$ commutes with both $G_n(I_0)$ and $G_n(I_0)^s$. This implies that there is a length-$2$ path in $Y_{n,m}$ from $(I_0,1)$ to $(I,\alpha)$ to $(I_0,s)$, as required for \ref{part:generatorconnect}.

We first handle the general case when $n\geq 2m+d$. Note that $G_n(I_0)\subset G_n(J\cup I_0)$ and $s\in G_n(J)\subset G_n(J\cup I_0)$. Therefore the conjugate $G_n(I_0)^s$ by $s$ is contained in $G_n(J\cup I_0)$. Since $\abs{J\cup I_0}\leq \abs{J}+\abs{I_0}=d+m$ and $n\geq 2m+d$, we can choose some $I\in \binom{[n]}{m}$ with $(J\cup I_0)\cap I=\emptyset$. By the definition of commuting \FIncgroup{}, there exists some $G_n$-conjugate $G_n(I)^\alpha$ of $G_n(I)$ commuting with $G_n(J\cup I_0)$. Since $G_n(I_0)\subset G_n(J\cup I_0)$, we know $G_n(I_0)$ commutes with $G_n(I)^\alpha$, and similarly $G_n(I_0)^s\subset G_n(J\cup I_0)$ implies $G_n(I_0)$ commutes with $G_n(I)^\alpha$. This provides the desired $(I,\alpha)$.

Now assume that $n\geq 2m+d-1$ and that $G$ satisfies~\eqref{eq:strongercommuting}. In this case we must choose $I_0=[d]$. We separate out two cases depending on whether $I_0$ and $J$ are disjoint. If $I_0\cap J\neq \emptyset$, then $\abs{I_0\cup J}\leq \abs{I_0}+\abs{J}-1=d+m-1$. Since $n\geq 2m+d-1$, as above we can find some $I\in \binom{[n]}{m}$ with $(J\cup I_0)\cap I=\emptyset$, so we find $(I,\alpha)$ as before.

The last case is when $I_0\cap J=\emptyset$. The condition~\eqref{eq:strongercommuting} then guarantees that $s\in G_n(J)$ commutes with $G_n(I_0)$, so $G_n(I_0)^s=G_n(I_0)$. Note that this does \emph{not} mean that there is an edge directly from $(I_0,1)$ to $(I_0,s)$! (In general $G_n(I_0)$ is nonabelian, so it does not commute with itself.) But since $n\geq 2m+1\geq 2m$ we can choose some $I\in \binom{[n]}{m}$ with $I\cap I_0=\emptyset$ and some conjugate $G_n(I)^\alpha$ commuting with $G_n(I_0)$, providing the desired $(I,\alpha)$.
\end{proof}

We now begin the induction with the base case $k=1$. Here there is almost nothing to prove. If $\abs{I}=\abs{I'}$ then $G_n(I)$ is isomorphic to $G_n(I')$, since an automorphism representing some $\gamma\in \Gamma_n$ takes one to a conjugate of the other. In particular, for each $I\in \binom{[n]}{d}$ the subgroup $G_n(I)$ is isomorphic to $G_n([d])$ which is a quotient of $G_d$. Since $G_d$ is finitely generated by assumption, so is $G_n(I)$. We have assumed that $G$ is generated in degree $d$, so for any $n\geq d$ the group $G_{n}$ is generated by the $\binom{n}{d}$ subgroups $G_{n}(I)$ for $I\in \binom{[n]}{d}$. Therefore $G_{n}=G_n[1]$ is finitely generated for all $n\geq d$.

Now fix for the rest of the proof some $n\geq (2(k+1)-1)d=(2k+1)d$ (or under the hypotheses of \ref{part:improverangebyone}, that $n\geq (2k+1)d-1$). We may assume by induction that $G_n[k]$ is finitely generated, and our goal is to prove that $G_n[k+1]$ is finitely generated.

To do this, we must prove that $S(G_n[k]/G_n[k+1])\subset \Sigma(G_n[k])$; in other words, for every nonzero homomorphism $\lambda\colon G_n[k]\to G_n[k]/G_n[k+1]\to \R$, we must show that $[\lambda]\in \Sigma(G_n[k])$.

Every such homomorphism factors uniquely as $G_n[k]\to L_n[k]\to \R$ for a unique map $L_n[k]\to \R$, which we denote by $\lambda$ as well. 
By Claim~\ref{claim:tiltgamma}, there exists some $\gamma\in \Gamma_n$ such that $\lambda'=\gamma^*\lambda$ has the property that $\lambda'|_{L_n[k](I)}\neq 0$ for all $I\in \binom{[n]}{dk}$. If $\widetilde{\gamma}$ is an automorphism of $G_n$ representing $\gamma\in \Gamma_n$, we have $\widetilde{\gamma}^*[\lambda]=[\lambda']\in S(G_n[k])$. Since $\Sigma(G_n[k])$ is invariant under automorphisms, it suffices to prove that $[\lambda']$ belongs to $\Sigma(G_n[k])$.

Therefore by possibly replacing $\lambda$ by $\lambda'$, we may assume that $\lambda$ has $\lambda|_{L_n[k](I)}\neq 0$ for all $I\in \binom{[n]}{dk}$. Since $L_n[k](I)$ is the image of $G_n[k](I)$ in $L_n[k]$, this means that for each $I$ there exists some $g_I\in G_n[k](I)$ with $\lambda(g_I)\neq 0$. Choose once and for all such an element $g_I\in G_n[k](I)$ for each $I \in \binom{[n]}{dk}$, and let $A=\{g_I^\alpha=\alpha^{-1}g_I\alpha\,|\,\alpha\in G_n,I \in \binom{[n]}{dk}\}$ be the set of all their $G_n$-conjugates. Since $\lambda$ factors through $G_n[k]/G_n[k+1]$, we have $\lambda(g_I^{\alpha})=\lambda(g_I)\neq 0$. Therefore $A$ survives under $\lambda$ in the sense of Proposition~\ref{prop:KMM}.

There is a natural map from the graph $Y_{n,dk}$ to the commuting graph $C(A)$ taking $(I,\alpha)\in Y_{n,dk}$ to $g_I^\alpha\in A$, which is surjective on vertices by definition. Every edge  in $Y_{n,dk}$ is taken to an edge in $C(A)$, since $g_I^\alpha$ is contained in the subgroup $G_n(I)^{\alpha}$ (so if $G_n(I)^\alpha$ and $G_n(I')^\beta$ commute then certainly $g_I^\alpha$ and $g_{I'}^\beta$ commute). Claim~\ref{claim:handlecomplex} states  under our assumptions on $n$ that $Y_{n,dk}$ is connected, so $C(A)$ is connected as well.

Since $n\geq (2k+1)d-1\geq dk$, Claim~\ref{claim:Gknormalgen} states that $G_n[k]$ is generated by the $G_n$-conjugates of the subgroups $G_n[k](I)$ for $I\in \binom{[n]}{dk}$. We let $B$ be this generating set $B=\bigcup_{I,\beta}G_n[k](I)^{\beta}$. It remains only to check that $A$ dominates $B$, so consider an arbitrary element $b\in B$ belonging to $G_n[k](I)^\beta$. Choose some vertex $(I',\alpha)$ of $Y_{n,dk}$ adjacent to the vertex $(I,\beta)$. Every element of $G_n[k](I)^\beta$ commutes with every element of $G_n[k](I')^\alpha$, so certainly $g$ commutes with $g_{I'}^\alpha\in A$. Thus every element of $B$ commutes with some element of $A$, so $A$ dominates $B$. This verifies the hypotheses of Proposition~\ref{prop:KMM}, so applying that proposition shows that $[\lambda]\in \Sigma(G_n[k])$. As discussed above, we conclude that $S(G_n[k]/G_n[k+1])$ is contained in $\Sigma(G_n[k])$. Therefore $G_n[k]$ is finitely generated. This concludes the proof of Theorem~\ref{maintheorem:stronger}.\qed

\section{\texorpdfstring{Deducing Theorems \ref{maintheorem:torellilcs} and \ref{maintheorem:ialcs} from Theorem~\ref{maintheorem:stronger}}{Deducing Theorems A and C from Theorem~E'}}
\label{section:examples}

In this section, we will show how to derive Theorems \ref{maintheorem:torellilcs} and \ref{maintheorem:ialcs} from
Theorem \ref{maintheorem:stronger}.  We begin in \S \ref{section:reductions} with some reductions.  The main body
of this derivation is in \S \ref{section:verify}.

\subsection{Reductions}
\label{section:reductions}

In this section we show that Theorem~\ref{maintheorem:torellilcs} it suffices to prove that $\Torelli_g^1[k]$ is finitely generated for $g \geq 6k-4$, and similarly for Theorem~\ref{maintheorem:ialcs}. We do this using the following two reductions.
\begin{compactitem}
\item If $\subgp$ is a group satisfying $\Torelli_g^b[k] \subset \subgp \subset \Torelli_g^b$, then letting
$\overline{\subgp}$ be the image of $\subgp$ in $\Torelli_g^b / \Torelli_g^b[k]$ we have a short
exact sequence
\[1 \longrightarrow \Torelli_g^b[k] \longrightarrow \subgp \longrightarrow \overline{\subgp} \longrightarrow 1.\]
Since $\Torelli_g^b / \Torelli_g^b[k]$ is a finitely generated nilpotent group, its subgroup
$\overline{\subgp}$ is also finitely generated.  To prove that $\subgp$ is finitely generated when $g \geq 6k-4$, it
is thus enough to prove that $\Torelli_g^b[k]$ is finitely generated for $g \geq 6k-4$.
\item The homomorphism $\Torelli_g^1 \rightarrow \Torelli_g$ obtained by gluing a disc
to $\partial \Sigma_g^1$ is surjective, and thus the restriction of this homomorphism
to $\Torelli_g^1[k]$ is also surjective.  To prove that $\Torelli_g[k]$ is finitely
generated for $g \geq 6k-4$, it is thus enough to prove that $\Torelli_g^1[k]$ is finitely generated for $g \geq 6k-4$.
\end{compactitem}
Using the first reduction, we similarly see that to prove Theorem \ref{maintheorem:ialcs} it is enough
to prove that $\IA_n[k]$ is finitely generated for $n \geq 6k-4$.

\subsection{\texorpdfstring{Applying Theorem~\ref{maintheorem:stronger} to $\Torelli$ and $\IA$}{Applying Theorem~E' to Torelli and IA}}
\label{section:verify}
It remains to show that $\Torelli_g^1[k]$ is finitely generated for $g \geq 6k-4$ and that
$\IA_n[k]$ is finitely generated for $n \geq 6k-4$.  We will deduce both of these from Theorem \ref{maintheorem:stronger}, so we must verify that its hypotheses are satisfied.

\para{Automorphism group of free group}
We showed in~\cite[Lemma~3.1]{ChurchPutmanGenJohnson} that there is an FI-group $G=\IA$ with $G_n=\IA_n$. This immediately implies that $\IA$ is an $\FInc$-group; indeed, restricting part (ii) of \cite[Definition~2.1]{ChurchPutmanGenJohnson} to inclusions yields precisely the definition of an $\FInc$-group. We will verify that the $\FInc$-group $\IA$ satisfies the hypotheses of Theorem~\ref{maintheorem:stronger} with $d=3$.

We can describe $\IA$ concretely. For any subset $I\subset \N$ let $F_I$ denote the free group on the set $\{x_i\,|\,i\in I\}$. Then $\IA_I$ is the subgroup of $\Aut(F_I)$ acting trivially on the abelianization. Whenever $I\subset K$ we obtain an injection $\IA_I\to \IA_K$ sending $\varphi\in \IA_I$ to the automorphism sending $x_i\mapsto\varphi(x_i)$ for $i\in I$ and sending $x_k\mapsto x_k$ for $k\notin I$. Since the maps $\IA_I\to \IA_K$ are injective we can canonically identify $\IA_I$ with $\IA_K(I)$.

From this description we see that the subgroups $\IA_I$ and $\IA_J$ of $\IA_K$ commute whenever $I\cap J=\emptyset$ (no $\IA_K$-conjugation is necessary). So $\IA$ is a commuting $\FInc$-group, and satisfies the condition~\eqref{eq:strongercommuting}.

Magnus~\cite{MagnusGenerators} proved that $\IA_n$ is finitely generated, and we deduced from Magnus's generating set in~\cite[Lemma~2.13 and Proposition 3.3]{ChurchPutmanGenJohnson} that $\IA$ is generated in degree $d=3$.

The extension $1\to \IA_n\to \Aut(F_n)\to \GL_n\Z\to 1$ gives an outer action of $\GL_n\Z$ on $\IA_n$. We take $\Gamma_n$ to be the subgroup $\Gamma_n=\SL_n\Z$. It was independently shown by Cohen--Pakianathan~\cite{CohenPakianathanH1IA}, Farb~\cite{FarbH1IA}, and Kawazumi~\cite{KawazumiH1IA} that $\HH_1(\IA_n;\R)\iso \Hom(H,\bwedge^2 H)$ where $H\iso \R^n$ is the standard representation of $\SL_n\Z$, so this does indeed extend to an action of $\algp_n=\SL_n$.

The subgroup of $\Aut(F_n)$ permuting the generators and their inverses is the signed permutation group $S_n^\pm$. If $\widetilde{\sigma}\in S_n^\pm$ projects to $\sigma\in S_n$, the explicit description of $\IA(F_K)$ above for $K\subset [n]$ shows that $\widetilde{\sigma}\in \Aut(F_n)$ conjugates $\IA(F_K)$ to $\IA(F_{\sigma(K)})$. Since the index-2 subgroup of $S_n^\pm$ projecting to $\SL_n\Z$ surjects onto $S_n$, their images in $\Gamma_n=\SL_n\Z$ provide the necessary elements $\gamma\in \Gamma_n$.

This verifies that the $\FInc$-group $G=\IA$ satisfies the hypotheses of Theorem~\ref{maintheorem:stronger} with $d=3$. We conclude from the structure theorem that the group $G_n[k]=\IA_n[k]$ is finitely generated for all $n\geq (2k-1)d-1=6k-4$.

\para{Mapping class group}
We defined in~\cite{ChurchPutmanGenJohnson} (in the paragraphs following Lemma 4.2) a weak FI-group $G=\Torelli$ with $G_n=\Torelli_n^1$. This immediately implies that $\Torelli$ is an $\FInc$-group; indeed, part (iii) of \cite[Definition~2.4]{ChurchPutmanGenJohnson} is precisely the definition of an $\FInc$-group.

We can be more concrete about the $\FInc$-group $\Torelli$. In \cite[\S4.1]{ChurchPutmanGenJohnson} we gave an infinite-genus surface $\Sigma_\N$, and for every $I\subset \N$ a surface $\Sigma_I\subset \Sigma_\N$ homeomorphic to $\Sigma_{\abs{I}}^1$, such that $I\subset J$ implies $\Sigma_I\subset \Sigma_J$. We can then define $\Torelli_I$ to be the subgroup $\Torelli(\Sigma_I)$ consisting of mapping classes supported on $\Sigma_I$ and acting trivially on homology. Since $\Sigma_I\subset \Sigma_K$ when $I\subset K$, we can identify $\Torelli_I$ with $\Torelli_K(I)$.

There is an important but subtle point here: for \emph{any} family of nested subsurfaces $\Sigma_I$, the description above would define a perfectly good $\FInc$-group $\Torelli$. However, it is only by carefully choosing the surfaces $\Sigma_I$ that we can guarantee that $\Torelli$ will be generated in degree 3. See \cite[Figure 2]{ChurchPutmanGenJohnson} for an illustration of the subsurfaces that we choose. In \cite[Proposition~4.5]{ChurchPutmanGenJohnson} we improved on Putman's cubic generating set from~\cite{PutmanSmallGenset} to show that with this definition, $\Torelli$ is indeed generated in degree 3. (Note incidentally that since we are not working with weak FI-groups anymore, the remark preceding \cite[Lemma 4.6]{ChurchPutmanGenJohnson} is no longer necessary.) Of course, Johnson~\cite{JohnsonFinite} proved that $\Torelli_3^1$ is finitely generated.

Given $I,J\subset [n]$, we can see from \cite[Figure 2]{ChurchPutmanGenJohnson} that $\Sigma_I$ and $\Sigma_J$ need not be disjoint even if $I\cap J=\emptyset$, so $\Torelli_I$ and $\Torelli_J$ need not commute.
However, we can always find a subsurface $\Sigma'_J$ homeomorphic to $\Sigma_J$ that is disjoint from $\Sigma_I$ and with $\HH_1(\Sigma'_J;\R)=\HH_1(\Sigma_J;\R)$ as subspaces of $\HH_1(\Sigma_{[n]};\R)$. By \cite[Lemma 4.1(ii)]{ChurchPutmanGenJohnson}, this implies that the subgroup $\Torelli_J=\Torelli(\Sigma_J)$ is $\Torelli_{[n]}$-conjugate  to $\Torelli(\Sigma'_J)$. Since the latter commutes with $\Torelli_I$, this verifies that $\Torelli$ is a commuting $\FInc$-group. Moreover, from the description in \cite[\S4.1 and Figure 2]{ChurchPutmanGenJohnson} one sees that $\Sigma_{[m]}$ \emph{is} disjoint from $\Sigma_K$ whenever $K\cap [m]=\emptyset$, so $\Torelli$ satisfies the condition~\eqref{eq:strongercommuting}. 

The extension $1\to \Torelli_n^1\to \Mod_n^1\to \Sp_{2n}\Z\to 1$ gives an outer action of $\Gamma_n=\Sp_{2n}\Z$ on $\Torelli_n^1$. Johnson's computation of $\HH_1(\Torelli_n^1;\Z)$ in~\cite{JohnsonAbel} implies that $\HH_1(\Torelli_n^1;\R)\iso \bwedge^3 H$, where $H\iso \R^{2n}$ is the standard representation of $\Sp_{2n}\Z$, so this does indeed extend to an action of $\algp_n=\Sp_{2n}$.

There is a natural subgroup of $\Sp_{2n}\Z$ isomorphic to $S_n$ which permutes the homology classes $\{a_1,\ldots,a_n\}$ and $\{b_1,\ldots,b_n\}$ (in the notation of \cite[Lemma 4.2]{ChurchPutmanGenJohnson}) separately. The action of $\sigma\in S_n$ on $\HH_1(\Sigma_{[n]})$ sends the subspace $\HH_1(\Sigma_K)$ to $\HH_1(\Sigma_{\sigma(K)})$. According to \cite[Lemma 4.1(ii)]{ChurchPutmanGenJohnson} this implies that this outer automorphism takes $\Torelli_K=\Torelli(\Sigma_K)$ to a $\Torelli_{[n]}$-conjugate of $\Torelli_{K'}=\Torelli(\Sigma_{K'})$. Therefore this subgroup of $\Gamma_n=\Sp_{2n}\Z$ provides the necessary elements $\gamma\in \Gamma_n$.

This verifies that the $\FInc$-group $G=\Torelli$ satisfies the hypotheses of Theorem~\ref{maintheorem:stronger} with $d=3$. We conclude from the structure theorem that the group $G_n[k]=\Torelli_n[k]$ is finitely generated for all $n\geq (2k-1)d-1=6k-4$.

\noindent
\begin{tabular*}{\linewidth}[t]{@{}p{\widthof{E-mail: {\tt tfchurchmathstanford.edu}}+0.75in}@{}p{\linewidth - \widthof{E-mail: {\tt tfchurchmathstanford.edu}} - 0.75in}@{}}
{\raggedright
Thomas Church\\
Department of Mathematics\\
Stanford University\\
450 Serra Mall\\
Stanford, CA 94305\\
E-mail: \myemail{tfchurch@stanford.edu}}
&
{\raggedright
Andrew Putman\\
Department of Mathematics\\
University of Notre Dame\\
279 Hurley Hall\\
Notre Dame, IN 46556\\
E-mail: \myemail{andyp@nd.edu}}
\end{tabular*}
\end{document}